\documentclass[12pt]{amsart}

\usepackage{amssymb}
\usepackage{bm}
\usepackage{graphicx}
\usepackage{psfrag}
\usepackage{hyperref}
\usepackage{color}
\usepackage{url}
\usepackage{algpseudocode}
\usepackage{fancyhdr}
\usepackage{fancyhdr}
\usepackage{array}
\usepackage{longtable}
\usepackage[table]{xcolor}
\usepackage{stmaryrd}

\definecolor{bleudefrance}{RGB}{205,205,250}
\definecolor{blizzardblue}{RGB}{235,235,250}
\hypersetup{colorlinks=true, linkcolor=blue, citecolor=red, urlcolor=wine}
\usepackage{xy}
\input xy
\xyoption{all}

\newtheorem*{maintheorem*}{Main Theorem}
\newtheorem{theorem}{Theorem}[section]

\newtheorem{prop}[theorem]{Proposition}
\newtheorem{conj}[theorem]{Conjecture}
\newtheorem{lemma}[theorem]{Lemma}
\newtheorem{cor}[theorem]{Corollary}

\theoremstyle{definition}
\newtheorem{definition}[theorem]{Definition}
\newtheorem{remark}[theorem]{Remark}
\newtheorem{example}[theorem]{Example}
\numberwithin{equation}{section}

\newcommand{\nn}{\mathbb{N}}

\newcommand{\qq}{\mathbb{Q}}
\newcommand{\rr}{\mathbb{R}}
\newcommand{\zz}{\mathbb{Z}}

\providecommand\llb{\llbracket}
\providecommand\rrb{\rrbracket}

\keywords{Puiseux monoids, factorization theory, factorization invariants, system of sets of lengths, realization theorem, set of distances, catenary degree, tame degree}


\begin{document}
	
	\mbox{}
	\title{Factorization invariants of Puiseux monoids \\ generated by geometric sequences}
	
	\author{Scott T. Chapman}
	\address{Department of Mathematics and Statistics\\Sam Houston State University\\Huntsville, TX 77341}
	\email{scott.chapman@shsu.edu}
	
	
	\author{Felix Gotti}
	\address{Department of Mathematics\\UC Berkeley\\Berkeley, CA 94720 \newline \indent Department of Mathematics\\Harvard University\\Cambridge, MA 02138}
	\email{felixgotti@berkeley.edu}
	\email{felixgotti@harvard.edu}
	
	\author{Marly Gotti}
	\address{Department of Mathematics\\University of Florida\\Gainesville, FL 32611}
	\email{marlycormar@ufl.edu}
	
	\date{\today}
	
	\begin{abstract}
	We study some of the factorization invariants of the class of Puiseux monoids generated by geometric sequences, and we compare and contrast them with the known results for numerical monoids generated by arithmetic sequences. The class we study here consists of all atomic monoids of the form $S_r := \langle r^n \mid n \in \nn_0 \rangle$, where $r$ is a positive rational. As the atomic monoids $S_r$ are nicely generated, we are able to give detailed descriptions of many of their factorization invariants. One distinguishing characteristic of $S_r$ is that all its sets of lengths are arithmetic sequences of the same distance, namely $|a-b|$, where $a,b \in \nn$ are such that $r = a/b$ and $\gcd(a,b) = 1$. We prove this, and then use it to study the elasticity and tameness of $S_r$.
	\end{abstract}
\maketitle

\tableofcontents

\section{Prologue} \label{sec:pro}

There is little argument that the study of factorization properties of rings and integral domains was a driving force in the early development of commutative algebra. Most of this work centered on determining when an algebraic structure has ``nice'' factorization properties (i.e., what today has been deemed a unique factorization domain (or UFD). It was not until the appearance of papers in the 1970s and 1980s by Skula~\cite{Sk76}, Zaks~\cite{Z80}, Narkiewicz~\cite{N71, N79}, Halter-Koch~\cite{HK84}, and Valenza\footnote{While Valenza's paper appeared in 1990, it was actually submitted 10 years earlier.}~\cite{V90} that there emerged interest in studying the deviation of an algebraic object from the UFD condition. Implicit in much of this work is the realization that problems involving factorizations of elements in a ring or integral domain are merely problems involving the multiplicative semigroup of the object in question. Hence, until the early part of the 21st century, many papers studying non-unique factorization were written from a purely multiplicative point of view (which to a large extent covered Krull domains and monoids). This changed with the appearance of~\cite{BCKR} and~\cite{CHM}, both of which view factorization problems in additive submonoids of the natural numbers known as \textit{numerical monoids}. These two papers generated a flood of work in this area, from both the pure \cite{ACHP07, CCMMP17, CDHK10, CGL09, CHK09, CKLNZ14, CK17, OP18} and the computational \cite{BOP17b, ACKT11, GGMV2, GOW19} points of view. Over the past three years, similar studies have emerged for additive submonoids of the nonnegative rational numbers, also known as \textit{Puiseux monoids} \cite{fG19a, fG17, fG18, fG19b, GG17, GO17}. 

The purpose of our work here is to highlight a class of Puiseux monoids with extremely nice factorization properties. This is in line with the earlier work done for numerical monoids. Indeed, several of the papers we have already cited are dedicated to showing that while general numerical monoids have complicated factorization properties, those that are generated by an arithmetic sequence have very predictable factorization invariants (see \cite{ACHP07, BCKR, CCMMP17, CGL09}). After fixing a positive rational $r>0$, we will study the additive submonoid of~$\qq_{\ge 0}$ generated by the set $\{r^n \mid n \in \nn_0 \}$. We denote this monoid by $S_r$, that is, $S_r := \langle r^n \mid n \in \nn_0 \rangle$ (cf. Definition \ref{semiring}). Observe that $S_r$ is also closed under multiplication and, therefore, it is a \textit{semiring}. Moreover, the semiring $S_r$ is \emph{cyclic}, which means that $S_r$ is generated as a semiring by only one element, namely $r$. We emphasize that when dealing with $S_r$, we will only be interested in factorizations with regard to its additive operation. However, we will use the term ``rational cyclic semiring" throughout this paper to represent the longer term ``Puiseux monoid generated by a geometric sequence."

We break our work into five sections. Our paper is self-contained and all necessary background and definitions can be found in Section~\ref{sec:bfactsdefs}. In Section~\ref{sec:set of length} we completely describe the structure of the sets of lengths in $S_r$, showing that such sets are always arithmetic progressions (Theorem~\ref{thm:sets of lengths}). In Section~\ref{sec:elasticity} we investigate the elasticity of~$S_r$ (Corollary \ref{thm:sets of lengths}) and explore in Propositions~\ref{accepted},~\ref{prop:set of elasticities}, and~\ref{local} the notions of accepted, full, and local elasticity. Finally, in Section~\ref{sec:tame degree} we study the omega primality of $S_r$ (Proposition \ref{prop:omega primality}), and use it to characterize the semirings $S_r$ that are locally and globally tame (Theorem \ref{thm:cyclic semirings are no locally tame}).

\section{Basic Facts and Definitions}
\label{sec:bfactsdefs}

In this section we review some of the standard concepts we shall be using later. The book~\cite{pG01} by Grillet provides a nice introduction to commutative monoids while the book~\cite{GH06} by Geroldinger and Halter-Koch offers extensive background in non-unique factorization theory of commutative domains and monoids. Throughout our exposition, we let $\mathbb{N}$ denote the set of positive integers, and we set $\nn_0 := \nn \cup \{0\}$. For $a, b \in \rr \cup \{\pm \infty \}$, let
\[
	\llb a, b \rrb = \{ x \in \zz \colon a \le x \le b \}
\]
be the discrete interval between $a$ and $b$. In addition, for $X \subseteq \rr$ and $r \in \rr$, we set
\[
	X_{> r} := \{x \in X \mid x > r\},
\]
and we use the notations $X_{< r}$ and $X_{\ge r}$ in a similar manner. If $q \in \qq_{> 0}$, then we call the unique $a,b \in \nn$ such that $q = a/b$ and $\gcd(a,b)=1$ the \emph{numerator} and \emph{denominator} of $q$ and denote them by $\mathsf{n}(q)$ and $\mathsf{d}(q)$, respectively.

\subsection{Atomic Monoids} The unadorned term \emph{monoid} always means commutative cancellative semigroup with identity and, unless otherwise specified, each monoid here is written additively. A monoid is called \emph{reduced} if its only unit (i.e., invertible element) is $0$. For a monoid~$M$, we let $M^\bullet$ denote the set $M \! \setminus \! \{0\}$. For the remainder of this section, let $M$ be a reduced monoid. For $x,w \in M$, we say that $x$ \emph{divides} $w$ \emph{in} $M$ and write $x \mid_M w$ provided that $w = x + y$ for some $y \in M$. An element $p \in M^\bullet$ is called \emph{prime} if $p \mid_M x+y$ for some $x,y \in M$ implies that either $p \mid_M x$ or $p \mid_M y$.

For $S \subseteq M$ we write $M = \langle S \rangle$ when $M$ is generated by $S$, that is, no submonoid of $M$ strictly contained in $M$ contains $S$. We say that $M$ is \emph{finitely generated} if it can be generated by a finite set. An element $a \in M^\bullet$ is called an \emph{atom} provided that for each pair of elements $x,y \in M$ such that $a = x+y$ either $x=0$ or $y=0$. It is not hard to verify that every prime element is an atom. The set of atoms of $M$ is denoted by $\mathcal{A}(M)$. Clearly, every generating set of $M$ must contain $\mathcal{A}(M)$. If $\mathcal{A}(M)$ generates $M$, then $M$ is called \emph{atomic}. On the other hand, $M$ is called \emph{antimatter} when $\mathcal{A}(M)$ is empty.

Every submonoid $N$ of $(\nn_0,+)$ is finitely generated and atomic. Since $N$ is reduced, $\mathcal{A}(N)$ is the unique minimal generating set of $N$. When $\nn_0 \setminus N$ is finite, $N$ is called \emph{numerical monoid}. It is not hard to check that every submonoid of $(\nn_0,+)$ is isomorphic to a numerical monoid. 
If $N$ is a numerical monoid, then the \emph{Frobenius number} of $N$, denoted by $\mathcal{F}(N)$, is the largest element in $\mathbb{N}_0 \setminus N$. For an introduction to numerical monoids see~\cite{GR09} and for some of their many applications see~\cite{AG16}.

A submonoid of $(\qq_{\ge 0},+)$ is called a \emph{Puiseux monoid}. In particular, every numerical monoid is a Puiseux monoid. However, Puiseux monoids might not be finitely generated nor atomic. For instance, $\langle 1/2^n \mid n \in \nn\rangle$ is a non-finitely generated Puiseux monoid with empty set of atoms. A Puiseux monoid is finitely generated if and only if it is isomorphic to a numerical monoid~\cite[Proposition~3.2]{fG17}. On the other hand, a Puiseux monoid $M$ is atomic provided that $M^\bullet$ does not have $0$ as a limit point~\cite[Theorem~3.10]{fG17} (cf. Proposition~\ref{prop:BF sufficient condition}).

\subsection{Factorization Invariants} The \emph{factorization monoid} of $M$ is the free commutative monoid on $\mathcal{A}(M)$ and is denoted by $\mathsf{Z}(M)$. The elements of $\mathsf{Z}(M)$ are called \emph{factorizations}. If $z = a_1 \dots a_n$ is a factorization of $M$ for some $a_1, \dots, a_n \in \mathcal{A}(M)$, then $n$ is called the \emph{length} of $z$ and is denoted by $|z|$. The unique monoid homomorphism $\phi \colon \mathsf{Z}(M) \to M$ satisfying $\phi(a) = a$ for all $a \in \mathcal{A}(M)$ is called the \emph{factorization homomorphism} of $M$. For each $x \in M$ the set
\[
	\mathsf{Z}(x) := \phi^{-1}(x) \subseteq \mathsf{Z}(M)
\]
is called the \emph{set of factorizations} of $x$, while the set
\[
	\mathsf{L}(x) := \{|z| : z \in \mathsf{Z}(x)\}
\]
is called the \emph{set of lengths} of $x$. If $\mathsf{L}(x)$ is a finite set for all $x \in M$, then $M$ is called a \emph{BF-monoid}. The following proposition gives a sufficient condition for a Puiseux monoid to be a BF-monoid.

\begin{prop} \cite[Proposition~4.5]{fG19a} \label{prop:BF sufficient condition}
	Let $M$ be a Puiseux monoid. If $0$ is not a limit point of $M^\bullet$, then $M$ is a BF-monoid.
\end{prop}
\noindent The \emph{system of sets of lengths} of $M$ is defined by
\[
	\mathcal{L}(M) := \{\mathsf{L}(x) \mid x \in M\}.
\]
The system of sets of lengths of numerical monoids has been studied in~\cite{ACHP07} and~\cite{GS18}, while the system of sets of lengths of Puiseux monoids was first studied in~\cite{fG19b}. In addition, a friendly introduction to sets of lengths and the role they play in factorization theory is surveyed in~\cite{aG16}. If $M$ is a BF-monoid and for each nonempty subset $S \subseteq \mathbb{N}_{\geq 2}$ there exists $x \in M$ with $\mathsf{L}(x) = S$, then we say that $M$ has the \emph{Kainrath property} (see \cite{K}). In a monoid with the Kainrath property, all possible sets of lengths are obtained.

For $x \in M^\bullet$, a positive integer $d$ is said to be a \emph{distance} of $x$ provided that the equality $\mathsf{L}(x) \cap \llb \ell, \ell + d \rrb = \{\ell, \ell + d\}$ holds for some $\ell \in \mathsf{L}(x)$. The set consisting of all the distances of $x$ is denoted by $\Delta(x)$ and called the \emph{delta set} of $x$. In addition, the set
\[
	\Delta(M) := \bigcup_{x \in M^\bullet} \Delta(x)
\]
is called the \emph{delta set} of the monoid $M$. The delta set of numerical monoids has been studied by the first author \emph{et al.} (see~\cite{BCKR, CKLNZ14} and references therein).

For two factorizations $z = \sum_{a \in \mathcal{A}(M)} \mu_a a$ and $z' = \sum_{a \in \mathcal{A}(M)} \nu_a a$ in $\mathsf{Z}(M)$, we set
\[
	\gcd(z,z') := \sum_{a \in \mathcal{A}(M)} \min\{\mu_a, \nu_a\} a,
\]
and we call the factorization $\gcd(z,z')$ the \emph{greatest common divisor} of $z$ and $z'$. In addition, we call
\[
	\mathsf{d}(z,z') := \max \big\{ |z| - |\gcd(z,z')|, |z'| - |\gcd(z,z')| \big\}
\]
the \emph{distance} between $z$ to $z'$ in $\mathsf{Z}(M)$. For $N \in \nn_0 \cup \{\infty\}$, a finite sequence $z_0, z_1, \dots, z_k$ in $\mathsf{Z}(x)$ is called an \emph{$N$-chain of factorizations} connecting $z$ and $z'$ if $z_0 = z$, $z_k = z'$, and $\mathsf{d}(z_{i-1}, z_i) \le N$ for $i \in \llb 1, k \rrb$. For $x \in M$, let $\mathsf{c}(x)$ denote the smallest $n \in \nn_0 \cup \{\infty\}$ such that for any two factorizations in $\mathsf{Z}(x)$ there exists an $n$-chain of factorizations connecting them. We call $\mathsf{c}(x)$ the \emph{catenary degree} of $x$ and we call
\[
	\mathsf{c}(M) := \sup \{ \mathsf{c}(x) \mid x \in M \} \in \nn_0 \cup \{ \infty \}
\]
the \emph{catenary degree} of $M$. In addition, the set
\[
	\mathsf{Ca}(M) := \{\mathsf{c}(x) \mid x \in M \ \text{and} \ \mathsf{c}(x) > 0\}
\]
is called the \emph{set of positive catenary degrees}. Recent studies of the catenary degree of numerical monoids can be found in~\cite{CCMMP17} and~\cite{OP18}.   

We offer the reader in Tables~\ref{Table 1} and~\ref{Table 2} a comparison of the known factorization properties between general numerical monoids and Puiseux monoids. Table~\ref{Table 1} considers traditionally global factorization properties whose roots reach back into commutative algebra. Table~\ref{Table 2} considers the computation of factorization invariants which have become increasingly popular over the past 20 years.  Definitions related to the omega invariant and the tame degree can be found in Section~\ref{sec:tame degree}.

{\footnotesize
 	 \renewcommand{\arraystretch}{1.5}
 	 \begin{table}[t] \caption{Monoidal Factorization Properties: Numerical vs. Puiseux Monoids}\label{Table 1}
  	\begin{tabular}{ | p{7cm} | p{7cm}  |} \hline
		 \rowcolor{bleudefrance} \rule{0pt}{20pt} \textbf{Let $N$ be a numerical monoid} & \textbf{Let $M$ be a Puiseux monoid} \\ \hline
		\rowcolor{blizzardblue}\multicolumn{2}{|c|}{Is it finitely generated?}\\ \hline
     	Always. & Not always: $M$ is finitely generated if and only if $M$ is isomorphic to a numerical monoid \cite[Prop.~3.2]{fG17}.  \\ \hline
		\rowcolor{blizzardblue}\multicolumn{2}{|c|}{Is it atomic?}\\ \hline
     	Always. & Not always: $\langle 1/2^n \mid n \in \mathbb{N} \rangle$ is not atomic. $M$ is atomic if $0$ is not a limit point of $M$ \cite[Thm.~3.10]{fG17}.
     	 \\ \hline
        \rowcolor{blizzardblue}\multicolumn{2}{|c|}{Is it a BF-monoid (BFM)?}\\ \hline
     	 Always \cite[Prop. 2.7.8]{GH06}. & Not always: $M$ can be atomic and not a BFM \cite[Ex.~5.7]{fG19a}.   \\ \hline
        \rowcolor{blizzardblue}\multicolumn{2}{|c|}{Is it an FF-monoid (FFM)?}\\ \hline
     	Always \cite[Prop. 2.7.8]{GH06}. & Not always: $M$ can be a BFM and not an FFM \cite[Ex.~4.9]{GO17}.
     	\\ \hline
		\rowcolor{blizzardblue}\multicolumn{2}{|c|}{Is it a Krull monoid?}\\ \hline
		 Not always: $N$ is a Krull monoid if and only if $N$ is isomorphic to $(\mathbb{N}_0, +)$ \cite[Thm.~5.5\,(2)]{GSZ17}. & 
		 Not always: $M$ is a Krull monoid if and only if $M$ is isomorphic to $(\mathbb{N}_0, +)$ \cite[Thm.~6.6]{fG18}. \\ \hline
		 \end{tabular} \end{table}
		 
		 
		 \begin{table}[t] \caption{Monoidal Factorization Invariants: Numerical vs. Puiseux Monoids}\label{Table 2}
		 	\begin{tabular}{ | p{7cm} | p{7cm}  |} \hline
		 \rowcolor{bleudefrance} \rule{0pt}{20pt} \textbf{Let $N$ be a numerical monoid} & \textbf{Let $M$ be a Puiseux monoid} \\ \hline
		\rowcolor{blizzardblue}\multicolumn{2}{|c|}{System of sets of lengths}\\ \hline
		Sets of lengths in $N$ are almost arithmetic progressions \cite[Thm.~4.3.6]{GH06}.
		Also, for $L \subseteq \mathbb{N}_{\ge 2}$, there is a numerical monoid $N$ and $x \in N$ with $\mathsf{L}(x) = L$ \cite[Thm.~3.3]{GS18}. &
		Sets of lengths can have arbitrary behavior as there exists a Puiseux monoid satisfying the Kainrath property \cite[Thm.~3.6]{fG19b}.
		 \\ \hline
		\rowcolor{blizzardblue}\multicolumn{2}{|c|}{Elasticity}\\ \hline		
		$\rho(N) = \frac{\max {\mathcal{A}(N)}}{\min {\mathcal{A}(N)}}$ is always finite and accepted \cite[Thm. 2.1]{CHM}.
		Moreover, $N$ is fully elastic if and only if $N$ is isomorphic to $(\mathbb{N}_0, +)$ \cite[Thm.~2.2]{CHM}. 
		& If $M$ is atomic, then $\rho(M) = \infty$ if $0$ is a limit point of $\mathcal{A}(M)$ and $\rho(M) = \frac{\sup {\mathcal{A}(M)}}{\inf {\mathcal{A}(M)}}$ otherwise \cite[Thm.~3.2]{GO17}.  Moreover, $\rho(M)$ is accepted if and only if $\mathcal{A}(M)$ has a minimum and a maximum in $\mathbb{Q}$ \cite[Thm.~3.4]{GO17}.
		\\ \hline
		\rowcolor{blizzardblue}\multicolumn{2}{|c|}{Catenary degree}\\ \hline
		$\mathsf{c}(N) \le \frac{\mathcal{F}(N) + \max{\mathcal{A}(N)}}{\min{\mathcal{A}(N)}} + 1$ \cite[Ex. 3.1.6]{GH06}. &
		No known general results.\\ \hline
		\rowcolor{blizzardblue}\multicolumn{2}{|c|}{Tame degree}\\ \hline
		Always globally tame (and, consequently, locally tame) \cite[Thm. 3.1.4]{GH06}. &
		Not always locally tame (see Theorem~\ref{thm:cyclic semirings are no locally tame}). \\ \hline
		\rowcolor{blizzardblue}\multicolumn{2}{|c|}{Omega primality}\\ \hline
		$\omega(N) < \infty$ always.
		&
		$\omega(S_r) = \infty$ when $r \in \qq \cap (0,1)$ and $\mathsf{n}(r) > 1$ (see Theorem~\ref{thm:cyclic semirings are no locally tame}).\\ \hline
  	\end{tabular} \end{table}
}  	
  	
\medskip
\subsection{Cyclic Rational Semirings} As mentioned in the introduction, in this paper we study factorization invariants of those Puiseux monoids that are generated as a semiring by a single element.

\begin{definition}\label{semiring}
	For $r \in \qq_{>0}$, we call \emph{cyclic rational semiring} to the Puiseux monoid~$S_r$ additively generated by the nonnegative powers of $r$, i.e., $S_r = \big\langle r^n \mid n \in \nn_0 \rangle$.
\end{definition}

Although no systematic study of the factorization of cyclic rational semirings has been carried out so far, in~\cite{GG17} the atomicity of $S_r$ was first considered and classified in terms of the parameter $r$, as the next result indicates.

\begin{theorem} \cite[Theorem~6.2]{GG17}\label{thm:atomic classification of multiplicative cyclic Puiseux monoids}
	For $r \in \qq_{> 0}$, let $S_r$ be the cyclic rational semiring generated by $r$. Then the following statements hold.
	\begin{enumerate}
		\item If $\mathsf{d}(r)=1$, then $S_r$ is atomic with $\mathcal{A}(S_r) = \{1\}$.
		\vspace{3pt}
		\item If $\mathsf{d}(r) > 1$ and $\mathsf{n}(r) = 1$, then $S_r$ is antimatter.
		\vspace{3pt}
		\item If $\mathsf{d}(r) > 1$ and $\mathsf{n}(r) > 1$, then $S_r$ is atomic with $\mathcal{A}(S_r) = \{r^n \mid n \in \nn_0\}$.
	\end{enumerate}
\end{theorem}

As a consequence of Theorem~\ref{thm:atomic classification of multiplicative cyclic Puiseux monoids}, the monoid $S_r$ is atomic precisely when $r \in \qq_{> 0}$ and either $r = 1$ or $\mathsf{n}(r) > 1$.
\bigskip

\section{Sets of Lengths Are Arithmetic Sequences}
\label{sec:set of length}

In this section we show that the set of lengths of each element in an atomic rational cyclic semiring $S_r$ is an arithmetic sequence. First, we describe the minimum-length and maximum-length factorizations for elements of $S_r$. We start with the case where $0 < r < 1$.

\begin{lemma} \label{lem:factorization of extremal length II}
	Take $r \in (0,1) \cap \qq$ such that $S_r$ is atomic, and for $x \in S_r^\bullet$ consider the factorization $z = \sum_{i=0}^N \alpha_i r^i \in \mathsf{Z}(x)$, where $N \in \nn$ and $\alpha_0, \dots, \alpha_N \in \nn_0$. The following statements~hold.
	\begin{enumerate}
		\item $\min \mathsf{L}(x) = |z|$ if and only if $\alpha_i < \mathsf{d}(r)$ for $i \in \llb 1, N \rrb$.
		\vspace{3pt}
		
		\item There exists exactly one factorization in $\mathsf{Z}(x)$ of minimum length.
		\vspace{3pt}
		
		\item $\sup \mathsf{L}(x) = \infty$ if and only if $\alpha_i \ge \mathsf{n}(r)$ for some $i \in \llb 0, N \rrb$.
		\vspace{3pt}
		
		\item $|\mathsf{Z}(x)| = 1$ if and only if $|\mathsf{L}(x)| = 1$, in which case, $\alpha_i < \mathsf{n}(r)$ for $i \in \llb 0, N \rrb$.
	\end{enumerate} 
\end{lemma}

\begin{proof}
	To verify the direct implication of~(1), we only need to observe that if $\alpha_i \ge \mathsf{d}(r)$ for some $i \in \llb 1,N \rrb$, then the identity $\alpha_i r^i = (\alpha_i - \mathsf{d}(r))r^i + \mathsf{n}(r)r^{i-1}$ would yield a factorization $z'$ in $\mathsf{Z}(x)$ with $|z'| < |z|$. To prove the reverse implication, suppose that $w := \sum_{i=0}^K \beta_i r^i \in \mathsf{Z}(x)$ has minimum length. By the implication already proved, $\beta_i < \mathsf{d}(r)$ for $i \in \llb 1, N \rrb$. Insert zero coefficients if necessary and assume that $K = N$. Suppose, by way of contradiction, that there exists $m \in \llb 1,N \rrb$ such that $\beta_m \neq \alpha_m$ and assume that such index $m$ is as large as possible. Since $z,w \in \mathsf{Z}(x)$ we can write
	\[
		(\alpha_m - \beta_m)r^m = \sum_{i=0}^{m-1} (\beta_i - \alpha_i) r^i.
	\]
	After multiplying the above equality by $\mathsf{d}(r)^m$, it is easy to see that $\mathsf{d}(r) \mid \alpha_m - \beta_m$, which contradicts the fact that $0 < |\alpha_m - \beta_m| \le \mathsf{d}(r)$. Hence $\beta_i = \alpha_i$ for $i \in \llb 0, N \rrb$ and, therefore, $w = z$. As a result, $|z| = |w| = \min \mathsf{L}(x)$. In particular, there exists only one factorization in $\mathsf{Z}(x)$ having minimum length, and~(2) follows.
	
	For the direct implication of~(3), take a factorization $w = \sum_{i=0}^N \beta_i r^i \in \mathsf{Z}(x)$ whose length is not the minimum of $\mathsf{L}(x)$; such a factorization exists because $\sup \mathsf{L}(x) =~\infty$. By part~(1), there exists $i \in \llb 1, N \rrb$ such that $\beta_i \ge \mathsf{d}(r)$. Now we can use the identity $\beta_i r^i = (\beta_i - \mathsf{d}(r))r^i + \mathsf{n}(r)r^{i-1}$ to obtain $w_1 \in \mathsf{Z}(x)$ with $|w_1| < |w|$. Notice that there is an atom (namely $r^{i-1}$) appearing at least $\mathsf{n}(r)$ times in $w_1$. In a similar way we can obtain factorizations $w = w_0, w_1, \dots, w_n$ in $\mathsf{Z}(x)$, where $w_n =: \sum_{i=0}^N \beta'_i r^i \in \mathsf{Z}(x)$ satisfies $\beta_i' < \mathsf{d}(r)$ for $i \in \llb 1, N \rrb$. By~(1) we have that $w_n$ is a factorization of minimum length and, therefore, $z = w_n$ by~(2). Hence $\alpha_i \ge \mathsf{n}(r)$ for some $i \in \llb 0,N \rrb$, as desired. For the reverse implication, it suffices to note that given a factorization $w = \sum_{i=0}^N \beta_i r^i \in \mathsf{Z}(x)$ with $\beta_i \ge \mathsf{n}(r)$ we can use the identity $\beta_i r^i = (\beta_i - \mathsf{n}(r)) r^i + \mathsf{d}(r) r^{i+1}$ to obtain another factorization $w' = \sum_{i=0}^{N+1} \beta'_i r^i \in \mathsf{Z}(x)$ (perhaps $\beta'_{N+1} = 0$) with $|w'| > |w|$ and satisfying $\beta_{i+1} > \mathsf{n}(r)$. 
	
Finally, we argue the reverse implication of~(4) as the direct implication is trivial. To do this, assume that $\mathsf{L}(x)$ is a singleton. Then each factorization of $x$ has minimum length. By~(2) there exists exactly one factorization of minimum length in~$\mathsf{Z}(x)$. Thus, $\mathsf{Z}(x)$ is also a singleton. The last statement of~(4) is straightforward.
\end{proof}

We continue with the case of $r>1$.

\begin{lemma} \label{lem:factorization of extremal length I}
	Take $r \in \qq_{> 1} \setminus \nn$ such that $S_r$ is atomic, and for $x \in S_r^\bullet$ consider the factorization $z = \sum_{i=0}^N \alpha_i r^i \in \mathsf{Z}(x)$, where $N \in \nn$ and $\alpha_0, \dots, \alpha_N \in \nn_0$. The following statements hold.
	\begin{enumerate}
		\item $\min \mathsf{L}(x) = |z|$ if and only if $\alpha_i < \mathsf{n}(r)$ for $i \in \llb 0, N \rrb$.
		\vspace{3pt}
		
		\item There exists exactly one factorization in $\mathsf{Z}(x)$ of minimum length.
		\vspace{3pt}
		
		\item $\max \mathsf{L}(x) = |z|$ if and only if $\alpha_i < \mathsf{d}(r)$ for $i \in \llb 1, N \rrb$.
		\vspace{3pt}
	
		\item There exists exactly one factorization in $\mathsf{Z}(x)$ of maximum length.
		\vspace{3pt}
		
		\item $|\mathsf{Z}(x)| = 1$ if and only if $|\mathsf{L}(x)| = 1$, in which case $\alpha_0 < \mathsf{n}(r)$ and $\alpha_i < \mathsf{d}(r)$ for $i \in \llb 1, N \rrb$.
	\end{enumerate} 
\end{lemma}

\begin{proof}
	To argue the direct implication of~(1) it suffices to note that if $\alpha_i \ge \mathsf{n}(r)$ for some $i \in \llb 0, N \rrb$, then we can use the identity $\alpha_i r^i = (\alpha_i - \mathsf{n}(r))r^i + \mathsf{d}(r)r^{i+1}$ to obtain a factorization $z'$ in $\mathsf{Z}(x)$ satisfying $|z'| < |z|$. For the reverse implication, suppose that $w = \sum_{i=0}^K \beta_i r^i$ is a factorization in $\mathsf{Z}(x)$ of minimum length. There is no loss in assuming that $K = N$. Note that $\beta_i < \mathsf{n}(r)$ for each $i \in \llb 0, N \rrb$ follows from the direct implication. Now suppose for a contradiction that $w \neq z$, and let $m$ be the smallest nonnegative integer satisfying that $\alpha_m \neq \beta_m$. Then
	\begin{equation} \label{eq:set length of rational semirings II}
		(\alpha_m - \beta_m) r^m = \sum_{i = m+1}^N (\beta_i - \alpha_i) r^i.
	\end{equation}
	After clearing the denominators in~(\ref{eq:set length of rational semirings II}), it is easy to see that $\mathsf{n}(r) \mid \alpha_m - \beta_m$, which implies that $\alpha_m = \beta_m$, a contradiction. Hence $w = z$ and so $|z| = |w| = \min \mathsf{L}(x)$. We have also proved that there exists a unique factorization of $x$ of minimum length, which is~(2).
	
	For the direct implication of~(3), it suffices to observe that if $\alpha_i \ge \mathsf{d}(r)$ for some $i \in \llb 1, N \rrb$, then we can use the identity $\alpha_i r^i = \big( \alpha_i - \mathsf{d}(r) \big) r^i + \mathsf{n}(r) r^{i-1}$ to obtain a factorization $z'$ in $\mathsf{Z}(x)$ satisfying $|z'| > |z|$. For the reverse implication of~(3), take $w = \sum_{i=0}^K \beta_i r^i$ to be a factorization in $\mathsf{Z}(x)$ of maximum length ($S_r$ is a BF-monoid by Proposition~\ref{prop:BF sufficient condition}). Once again, there is no loss in assuming that $K = N$. The maximality of $|w|$ now implies that $\beta_i < \mathsf{d}(r)$ for $i \in \llb 1, N \rrb$. Suppose, by way of contradiction, that $z \neq u$. Then take $m$ be the smallest index such that $\alpha_m \neq \beta_m$. Clearly, $m \ge 1$ and
	\begin{equation} \label{eq:set length of rational semirings I}
		(\alpha_m - \beta_m) r^m = \sum_{i=0}^{m-1} (\beta_i - \alpha_i) r^i.
	\end{equation}
	After clearing denominators, it is easy to see that $\mathsf{d}(r) \mid \alpha_m - \beta_m$, which contradicts that $0 < |\alpha_M - \beta_M| < \mathsf{d}(r)$. Hence $\alpha_i = \beta_i$ for each $i \in \llb 1, N \rrb$, which implies that $z = w$. Thus, $\max \mathsf{L}(x) = |z|$. In particular, there exists only one factorization of $x$ of maximum length, which is condition~(4).
	
	The direct implication of~(5) is trivial. For the reverse implication of~(5), suppose that $\mathsf{L}(x)$ is a singleton. Then any factorization in $\mathsf{Z}(x)$ is a factorization of minimum length. Since we proved in the first paragraph that $\mathsf{Z}(x)$ contains only one factorization of minimum length, we have that $\mathsf{Z}(x)$ is also a singleton. The last statement of~(5) is an immediate consequence of~(1) and~(3).
\end{proof}

We are in a position now to describe the sets of lengths of any atomic cyclic rational semiring.

\begin{theorem} \label{thm:sets of lengths}
	Take $r \in \qq_{>0}$ such that $S_r$ is atomic.
	\begin{enumerate}
		\item If $r < 1$, then for each $x \in S_r$ with $|\mathsf{Z}(x)| > 1$,
		\[
			\mathsf{L}(x) = \big\{ \min \mathsf{L}(x) + k \big( \mathsf{d}(r) - \mathsf{n}(r)\big) \mid k \in \nn_0 \big\}.
		\]
		\item If $r \in \nn$, then $|\mathsf{Z}(x)| = |\mathsf{L}(x)| = 1$ for all $x \in S_r$.
		\vspace{3pt}
		\item If $r \in \qq_{> 1} \setminus \nn$, then for each $x \in S_r$ with $|\mathsf{Z}(x)| > 1$,
		\[
			\mathsf{L}(x) = \bigg\{ \min \mathsf{L}(x) + k \big( \mathsf{n}(r) - \mathsf{d}(r) \big) \ \bigg{|} \ 0 \le k \le \frac{\max \mathsf{L}(x) - \min \mathsf{L}(x)}{\mathsf{n}(r) - \mathsf{d}(r)} \bigg\}.
		\]
	\end{enumerate}
Thus, $\mathsf{L}(x)$ is an arithmetic progression with difference $|\mathsf{n}(r) - \mathsf{d}(r)|$ for all $x \in S_r$.
\end{theorem}

\begin{proof}
	To argue~(1), take $x \in S_r$ such that $|\mathsf{Z}(x)| > 1$. Let $z := \sum_{i=0}^N \alpha_i r^i$ be a factorization in $\mathsf{Z}(x)$ with $|z| > \min \mathsf{L}(x)$. Lemma~\ref{lem:factorization of extremal length II} guarantees that $\alpha_i \ge \mathsf{d}(r)$ for some $i \in \llb 1, N \rrb$. Then one can use the identity $\alpha_i r^i = (\alpha_i - \mathsf{d}(r))r^i + \mathsf{n}(r)r^{i-1}$ to find a factorization $z_1 \in \mathsf{Z}(x)$ with $|z_1| = |z| - (\mathsf{d}(r) - \mathsf{n}(r))$. Carrying out this process as many times as necessary, we can obtain a sequence $z_1, \dots, z_n \in \mathsf{Z}(x)$, where $z_n =: \sum_{i=0}^K \alpha'_i r^i$ satisfies that $\alpha'_i < \mathsf{d}(r)$ for $i \in \llb 1, K \rrb$ and $|z_j| = |z| - j(\mathsf{d}(r) - \mathsf{n}(r))$ for $j \in \llb 1,n \rrb$. By Lemma~\ref{lem:factorization of extremal length II}(1), the factorization $z_n$ has minimum length and, therefore, $|z| \in \{ \min \mathsf{L}(x) + k \big( \mathsf{d}(r) - \mathsf{n}(r)\big) \mid k \in \nn_0 \}$. Then
	\begin{align} \label{eq:sets of lengths are arithmetic sequences 1}
		\mathsf{L}(x) \subseteq \big\{ \min \mathsf{L}(x) + k \big( \mathsf{d}(r) - \mathsf{n}(r)\big) \mid k \in \nn_0 \big\}.
	\end{align}
	For the reverse inclusion, we check inductively that $\min \mathsf{L}(x) + k (\mathsf{d}(r) - \mathsf{n}(r)) \in \mathsf{L}(x)$ for every $k \in \nn_0$. Since $|\mathsf{Z}(x)| > 1$, Lemma~\ref{lem:factorization of extremal length II}(2) guarantees that $|\mathsf{L}(x)| > 1$. Then there exists a factorization of length strictly greater than $\min \mathsf{L}(x)$, and we have already seen that such a factorization can be connected to a minimum-length factorization of $\mathsf{Z}(x)$ by a chain of factorizations in $\mathsf{Z}(x)$ with consecutive lengths differing by $\mathsf{d}(r) - \mathsf{n}(r)$. Therefore $\min \mathsf{L}(x) + (\mathsf{d}(r) - \mathsf{n}(r)) \in \mathsf{L}(x)$. Suppose now that $z = \sum_{i=0}^N \beta_i r^i$ is a factorization in $\mathsf{Z}(x)$ with length $\min \mathsf{L}(x) + k(\mathsf{d}(r) - \mathsf{n}(r))$ for some $k \in \nn$. Then by Lemma~\ref{lem:factorization of extremal length II}(1), there exists $i \in \llb 1,N \rrb$ such that $\beta_i \ge \mathsf{d}(r) > \mathsf{n}(r)$. Now using the identity $\beta_i r^i = (\beta_i - \mathsf{n}(r))r^i + \mathsf{d}(r)r^{i+1}$, one can produce a factorization $z' \in \mathsf{Z}(x)$ such that $|z'| = \min \mathsf{L}(x) + (k+1)(\mathsf{d}(r) - \mathsf{n}(r))$. Hence the reverse inclusion follows by induction.
	
	Clearly, statement~(2) is a direct consequence of the fact that $r \in \nn$ implies that $S_r = (\nn_0,+)$.
	
	To prove~(3), take $x \in S^\bullet_r$. Since $S_r$ is a BF-monoid, there exists $z \in \mathsf{Z}(x)$ such that $|z| = \max \mathsf{L}(x)$. Take $N \in \nn$ and $\alpha_0, \dots, \alpha_N \in \nn_0$ such that $z = \sum_{i=0}^N \alpha_i r^i$. If $\alpha_i \ge \mathsf{n}(r)$ for some $i \in \llb 0, N \rrb$, then we can use the identity $\alpha_i r^i = (\alpha_i - \mathsf{n}(r))r^i + \mathsf{d}(r)r^{i+1}$ to find a factorization $z_1 \in \mathsf{Z}(x)$ such that $|z_1| = |z| - (\mathsf{n}(r) - \mathsf{d}(r))$. Carrying out this process as many times as needed, we will end up with a sequence $z_1, \dots, z_n \in \mathsf{Z}(x)$, where $z_n =: \sum_{i=0}^K \beta_i r^i$ satisfies that $\beta_i < \mathsf{n}(r)$ for $i \in \llb 0, K \rrb$ and $|z_j| = |z| - j(\mathsf{n}(r) - \mathsf{d}(r))$ for $j \in \llb 1, n \rrb$. Lemma~\ref{lem:factorization of extremal length I}(1) ensures that $|z_n| = \min \mathsf{L}(x)$. Then
	\begin{align} \label{eq:sets of lengths are arithmetic sequences 2}
		\bigg\{ \min \mathsf{L}(x) + j( \mathsf{n}(r) - \mathsf{d}(r) ) \ \bigg{|} \ 0 \le j \le \frac{ \max \mathsf{L}(x) - \min \mathsf{L}(x) }{ \mathsf{n}(r) - \mathsf{d}(r) } \bigg\} \subseteq \mathsf{L}(x).
	\end{align}
	On the other hand, we can connect any factorization $w \in \mathsf{Z}(x)$ to the minimum-length factorization $w' \in \mathsf{Z}(x)$ by a chain $w = w_1, \dots, w_t = w'$ of factorizations in $\mathsf{Z}(x)$ so that $|w_i|- |w_{i+1}| = \mathsf{n}(r) - \mathsf{d}(r)$. As a result, both sets involved in the inclusion~(\ref{eq:sets of lengths are arithmetic sequences 2}) are indeed equal.
\end{proof}

We conclude this section collecting some immediate consequences of Theorem~\ref{thm:sets of lengths}.

\begin{cor} \label{bigcor} Take $r \in \qq_{>0}$ such that $S_r$ is atomic.
	\begin{enumerate}
		\item $S_r$ is a BF-monoid if and only if $r \ge 1$.
		\vspace{3pt}
		
		\item If $r \in \nn$, then $S_r \cong \nn_0$ and, as a result, $\Delta(x) = \emptyset$ and $\mathsf{c}(x) = 0$ for all $x \in S_r^\bullet$.
		\vspace{3pt}
		
		\item If $r \notin \nn$, then $\Delta(x) = \{ |\mathsf{n}(r) - \mathsf{d}(r)| \}$ for all $x \in S_r$ such that $|\mathsf{Z}(x)| > 1$. Therefore $\Delta(S_r) = \{ |\mathsf{n}(r) - \mathsf{d}(r)| \}$.
		\vspace{3pt}
		
		\item If $r \notin \nn$, then $\mathsf{Ca}(S_r) = \max \{\mathsf{n}(r), \mathsf{d}(r) \}$. Therefore $\mathsf{c}(S_r) = \max\{\mathsf{n}(r), \mathsf{d}(r)\}$.
	\end{enumerate}
\end{cor}

\begin{remark}
	Note that Corollary~\ref{bigcor}(4) contrasts with \cite[Theorem~4.2]{OP18} and \cite[Proposition~4.3.1]{GZ19}, where it is proved that most subsets of $\nn_0$ can be realized as the set of catenary degrees of a numerical monoid and a Krull monoid (finitely generated with finite class group), respectively.
\end{remark}
\section{The Elasticity}
\label{sec:elasticity}

\subsection{The Elasticity} An important factorization invariant related with the sets of lengths of an atomic monoid is the elasticity. Let $M$ be a reduced atomic monoid. The \emph{elasticity} of an element $x \in M^\bullet$, denoted by $\rho(x)$, is defined as
\[
	\rho(x) = \frac{\sup \mathsf{L}(x)}{\inf \mathsf{L}(x)}.
\]
By definition, $\rho(0) = 1$. Note that $\rho(x) \in \qq_{\ge 1} \cup \{\infty\}$ for all $x \in M^\bullet$. On the other hand, the \emph{elasticity} of the whole monoid $M$ is defined to be
\[
	\rho(M) := \sup \{\rho(x) \mid x \in M^\bullet\}.
\]
The elasticity was introduced by R.~Valenza~\cite{V90} as a tool to measure the phenomenon of non-unique factorizations in the context of algebraic number theory. The elasticity of numerical monoids has been successfully studied in~\cite{CHM}. In addition, the elasticity of atomic monoids naturally generalizing numerical monoids has received substantial attention in the literature in recent years (see, for instance,~\cite{fG19c,GO17,mG19,qZ18}). In this section we focus on aspects of the elasticity of cyclic rational semirings, sharpening for them some of the results established in~\cite{GO17} and~\cite{mG19}.

The following formula for the elasticity of an atomic Puiseux monoid in terms of the infimum and supremum of its set of atoms was established in~\cite{GO17}.

\begin{theorem} \cite[Theorem~3.2]{GO17} \label{thm:elasticity of PM}
	Let $M$ be an atomic Puiseux monoid. If $0$ is a limit point of $M^\bullet$, then
	$\rho(M) = \infty$. Otherwise,
	\[
		\rho(M) = \frac{\sup \mathcal{A}(M)}{\inf \mathcal{A}(M)}.
	\]
\end{theorem}

The next result is an immediate consequence of Theorem~\ref{thm:elasticity of PM}.

\begin{cor} \label{cor:elasticity of rational semirings}
	Take $r \in \qq_{>0}$ such that $S_r$ is atomic. Then the following statements are equivalent:
	\begin{enumerate}
		\item[(1)] $r \in \nn$;
		\vspace{3pt}
		
		\item[(2)] $\rho(S_r) = 1$;
		\vspace{3pt}
		
		\item[(3)] $\rho(S_r)<\infty$.
	\end{enumerate}
	Hence, if $S_r$ is atomic, then either $\rho(S_r)=1$ or $\rho(S_r)=\infty$.
\end{cor}

\begin{proof}
	To prove that (1) implies (2), suppose that $r \in \nn$. In this case, $S_r \cong \nn_0$. Since~$\nn_0$ is a factorial monoid, $\rho(S_r) = \rho(\nn_0) = 1$.  Clearly, (2) implies (3). Now assume (3) and that $r \notin \nn$. If $r < 1$, then $0$ is a limit point of $S_r^\bullet$ as $\lim_{n \to \infty} r^n = 0$. Therefore it follows by Theorem~\ref{thm:elasticity of PM} that $\rho(S_r) = \infty$. If $r > 1$, then $\lim_{n \to \infty} r^n = \infty$ and, as a result, $\sup \mathcal{A}(S_r) = \infty$. Then Theorem~\ref{thm:elasticity of PM} ensures that $\rho(S_r) = \infty$. Thus, (3) implies~(1). The final statement now easily follows.
\end{proof}

The elasticity of an atomic monoid $M$ is said to be \emph{accepted} if there exists $x \in M$ such that $\rho(M) = \rho(x)$.

\begin{prop}\label{accepted}
	Take $r \in \qq_{> 0}$ such that $S_r$ is atomic. Then the elasticity of $S_r$ is accepted if and only if $r \in \nn$ or $r < 1$.
\end{prop}

\begin{proof}
	For the direct implication, suppose that $r \in \qq_{> 1} \setminus \nn$. Corollary~\ref{cor:elasticity of rational semirings} ensures that $\rho(S_r) = \infty$. However, as $0$ is not a limit point of $S_r^\bullet$, it follows by Proposition~\ref{prop:BF sufficient condition} that~$S_r$ is a BF-monoid, and, therefore, $\rho(x) < \infty$ for all $x \in S_r$. As a result, $S_r$ cannot have accepted elasticity
	
	For the reverse implication, assume first that $r \in \nn$ and, therefore, that $S_r = \nn_0$. In this case, $S_r$ is a factorial monoid and, as a result, $\rho(S_r) = \rho(1) = 1$. Now suppose that $r < 1$. Then it follows by Corollary~\ref{cor:elasticity of rational semirings} that $\rho(S_r) = \infty$. In addition, for $x = \mathsf{n}(r) \in S_r$ Lemma~\ref{lem:factorization of extremal length II}(1) and Theorem~\ref{thm:atomic classification of multiplicative cyclic Puiseux monoids}(1) guarantee that
	\[
		\mathsf{L}(x) = \big\{ \mathsf{n}(r) + k \big( \mathsf{d}(r) - \mathsf{n}(r) \big) \mid k \in \nn_0 \big\}.
	\]
	Because $\mathsf{L}(x)$ is an infinite set, we have that $\rho(S_r) = \infty = \rho(x)$. Hence $S_r$ has accepted elasticity, which completes the proof.
\end{proof}

\subsection{The Set of Elasticities} For an atomic monoid $M$ the set
\[
	R(M) = \{ \rho(x) \mid x \in M\}
\]
is called the \emph{set of elasticities} of $M$, and $M$ is called \emph{fully elastic} if $R(M) = \qq \cap [1, \rho(M)]$ when $\infty \notin R(M)$ and $R(M) \setminus \{\infty\} = \qq \cap [1, \infty)$ when $\infty \in R(M)$. Let us proceed to describe the sets of elasticities of atomic cyclic rational semirings.

\begin{prop} \label{prop:set of elasticities}
	Take $r \in \qq_{> 0}$ such that $S_r$ is atomic.
	\begin{enumerate}
		\item If $r < 1$, then $R(S_r) = \{1, \infty\}$ and, therefore, $S_r$ is not fully elastic.
		\vspace{3pt}
		
		\item If $r \in \nn$, then $R(S_r) = \{1\}$ and, therefore, $S_r$ is fully elastic.
		\vspace{3pt}
		
		\item If $r \in \qq_{> 0} \setminus \nn$ and $\mathsf{n}(r) = \mathsf{d}(r) + 1$, then $S_r$ is fully elastic, in which case $R(S_r) = \qq_{\ge 1}$.
	\end{enumerate}
\end{prop}

\begin{proof}
	First, suppose that $r < 1$. Take $x \in S_r$ such that $|\mathsf{Z}(x)| > 1$. It follows by Theorem~\ref{thm:sets of lengths}(1) that $\mathsf{L}(x)$ is an infinite set, which implies that $\rho(x) = \infty$. As a result, $\rho(S_r) = \{1,\infty\}$ and then $S_r$ is not fully elastic.
	
	To argue~(2), it suffices to observe that $r \in \nn$ implies that $S_r = (\nn_0,+)$ is a factorial monoid and, therefore, $\rho(S_r) = \{1\}$.
	
	Finally, let us argue that $S_r$ is fully elastic when $\mathsf{n}(r) = \mathsf{d}(r) + 1$. To do so, fix $q \in \qq_{>1}$. Take $m \in \nn$ such that $m \mathsf{d}(q) > \mathsf{d}(r)$, and set $k = m \big( \mathsf{n}(q) - \mathsf{d}(q) \big)$. Let $t = m \mathsf{d}(q) - \mathsf{d}(r)$, and consider the factorizations $z = \mathsf{d}(r) r^k + \sum_{i=1}^{t} r^{k+i} \in \mathsf{Z}(S_r)$ and $z' = \mathsf{d}(r) \cdot 1 + \sum_{i=0}^{k-1} r^i  + \sum_{i=1}^{t} r^{k+i} \in \mathsf{Z}(S_r)$. Since $\mathsf{n}(r) = \mathsf{d}(r) + 1$, it can be easily checked that $\frac{1}{r-1} = \mathsf{d}(r)$. As
	\[
		\mathsf{d}(r) + \sum_{i=0}^{k-1} r^i + \sum_{i=1}^{t} r^{k+i}= \mathsf{d}(r) + \frac{r^k - 1}{r-1} + \sum_{i=1}^{t} r^{k+i} = \mathsf{d}(r)r^k + \sum_{i=1}^{t} r^{k+i},
	\]
	there exists $x \in S_r$ such that $z,z' \in \mathsf{Z}(x)$. By Lemma~\ref{lem:factorization of extremal length I} it follows that $z$ is a factorization of $x$ of minimum length and $z'$ is a factorization of $x$ of maximum length. Thus, 
	\[
		\rho(x) = \frac{|z'|}{|z|} = \frac{\mathsf{d}(r) + k + t}{\mathsf{d}(r) + t} = \frac{m \, \mathsf{n}(q)}{m \, \mathsf{d}(q)} = q.
	\]
	As $q$ was arbitrarily taken in $\qq_{>1}$, it follows that $R(S_r) = \qq_{\ge 1}$. Hence $S_r$ is fully elastic when $\mathsf{n}(r) = \mathsf{d}(r) + 1$.
\end{proof}
\medskip

We were unable to determine in Proposition~\ref{prop:set of elasticities} whether $S_r$ is fully elastic when $r \in \qq_{> 1} \setminus \nn$ with $\mathsf{n}(r) \neq \mathsf{d}(r) + 1$. However, we prove in Proposition \ref{prop:the set of elasticity of S_r is dense when r > 1} that the set of elasticities of $S_r$ is dense in $\rr_{\ge 1}$.

\begin{prop} \label{prop:the set of elasticity of S_r is dense when r > 1}
	If $r \in \qq_{>1} \setminus \nn$, then the set $R(S_r)$ is dense in $\rr_{\ge 1}$.
\end{prop}

\begin{proof}
	Since $\sup \mathcal{A}(S_r) = \infty$, it follows by Theorem~\ref{thm:elasticity of PM} that $\rho(S_r) = \infty$. This, along with the fact that $S_r$ is a BF-monoid (because of Proposition~\ref{prop:BF sufficient condition}), ensures the existence of a sequence $\{x_n\}$ of elements of $S_r$ such that $\lim_{n \to \infty} \rho(x_n) = \infty$. Then it follows by~\cite[Lemma~5.6]{GO17} that the set
	\[
		S := \bigg\{ \frac{\mathsf{n}(\rho(x_n)) + k}{\mathsf{d}(\rho(x_n)) + k} \ \bigg{|} \ n,k \in \nn \bigg\}
	\]
	is dense in $\rr_{\ge 1}$. Fix $n,k \in \nn$. Take $m \in \nn$ such that $r^m$ is the largest atom dividing $x_n$ in $S_r$. Now take $K := k \gcd(\min \mathsf{L}(x_n), \max \mathsf{L}(x_n))$. Consider the element $y_{n,k} := x_n + \sum_{i=1}^K r^{m + i} \in S_r$. It follows by Lemma~\ref{lem:factorization of extremal length I} that $x_n$ has a unique minimum-length factorization and a unique maximum-length factorization; let them be $z_0$ and $z_1$, respectively. Now consider the factorizations $w_0 := z_0 + \sum_{i=1}^K r^{m + i} \in \mathsf{Z}(y_{n,k})$ and $w_1 := z_1 + \sum_{i=1}^K r^{m + i} \in \mathsf{Z}(y_{n,k})$. Once again, we can appeal to Lemma~\ref{lem:factorization of extremal length I} to ensure that $w_0$ and $w_1$ are the minimum-length and maximum-length factorizations of $y_{n,k}$. Therefore $\min \mathsf{L}(y_{n,k}) = \min \mathsf{L}(x_n) + K$ and $\max \mathsf{L}(y_{n,k}) = \max \mathsf{L}(x_n) + K$. Then we have
	\[
		\rho(y_{n,k}) = \frac{\max \mathsf{L}(y_{n,k})}{\min \mathsf{L}(y_{n,k})} = \frac{\max \mathsf{L}(x_n) + K}{\min \mathsf{L}(x_n) + K} = \frac{\mathsf{n}(\rho(x_n)) + k}{\mathsf{d}(\rho(x_n)) + k}.
	\]
	Since $n$ and $k$ were arbitrarily taken, it follows that $S$ is contained in $R(S_r)$. As $S$ is dense in $\rr_{\ge 1}$ so is $R(S_r)$, which concludes our proof.
\end{proof}

\begin{cor}
	The set of elasticities of $S_r$ is dense in $\rr_{\ge 1}$ if and only if $r \in \qq_{> 1} \setminus \nn$.
\end{cor}

\begin{remark}
	Proposition~\ref{prop:the set of elasticity of S_r is dense when r > 1} contrasts with the fact that the elasticity of a numerical monoid is always nowhere dense in $\rr$~\cite[Corollary~2.3]{CHM}.
\end{remark}

Wishing to have a full picture of the sets of elasticities of cyclic rational semirings, we propose the following conjecture.

\begin{conj}
	For $r \in \qq_{>1} \setminus \nn$ such that $\mathsf{n}(r) > \mathsf{d}(r) + 1$, the monoid $S_r$ is fully elastic.
\end{conj}

\subsection{Local Elasticities and Unions of Sets of Lengths} For a nontrivial reduced monoid $M$ and $k \in \nn$, we let $\mathcal{U}_k(M)$ denote the union of sets of lengths containing $k$, that is, $\mathcal{U}_k(M)$ is the set of $\ell \in \nn$ for which there exist atoms $a_1, \dots, a_k, b_1, \dots, b_\ell$ such that $a_1 \dots a_k = b_1 \dots b_\ell$. The set $\mathcal{U}_k(M)$ is known as the \emph{union of sets of lengths} of $M$ containing~$k$. In addition, we set
\[
	 \lambda_k(M) := \min \, \mathcal{U}_k(M) \quad \text{and} \quad  \rho_k(M) := \sup \, \mathcal{U}_k(M),
\]
and we call $\rho_k(M)$ the \emph{k-th local elasticity} of $M$.
Unions of sets of lengths have received a great deal of attention in recent literature; see, for example, \cite{BS18,BGG11,FGKT17,sT19}. In particular, the unions of sets of lengths and the local elasticities of Puiseux monoids have been considered in~\cite{mG19}. By~\cite[Section~1.4]{GH06}, the elasticity of an atomic monoid can be expressed in terms of its local elasticities as follows
\[
	\rho(M) = \sup \bigg\{ \frac{\rho_k(M)}{k} \ \bigg{|} \ k \in \nn \bigg\} = \lim_{k \to \infty} \frac{\rho_k (M)}{k}.
\]

Let us conclude this section studying the unions of sets of lengths and the local elasticities of atomic cyclic rational semirings.

\begin{prop}\label{local}
	Take $r \in \qq_{> 0}$ such that $S_r$ is atomic. Then $\mathcal{U}_k(S_r)$ is an arithmetic progression containing $k$ with distance $|\mathsf{n}(r) - \mathsf{d}(r)|$ for every $k \in \nn$. More specifically, the following statements hold.
	\begin{enumerate}
		\item If $r < 1$, then
		\begin{itemize}
			\item $\mathcal{U}_k(S_r) = \{k\}$ if $k < \mathsf{n}(r)$,
			\vspace{2pt}
			\item $\mathcal{U}_k(S_r) = \{ k + j (\mathsf{d}(r) - \mathsf{n}(r)) \mid j \in \nn_0 \}$ if $\mathsf{n}(r) \le k < \mathsf{d}(r)$, and
			\vspace{2pt}
			\item $\mathcal{U}_k(S_r) = \{ k + j (\mathsf{d}(r) - \mathsf{n}(r)) \mid j \in \zz_{\ge \ell} \}$ for some $\ell \in \zz_{< 0}$ if $k \ge \mathsf{d}(r)$.
		\end{itemize}
		\vspace{3pt}
		
		\item If $r \in \qq_{> 1} \setminus \nn$, then
		\begin{itemize}
			\item $\mathcal{U}_k(S_r) = \{k\}$ if $k < \mathsf{d}(r)$,
			\vspace{2pt}
			\item $\mathcal{U}_k(S_r) = \{ k + j (\mathsf{n}(r) - \mathsf{d}(r)) \mid j \in \nn_0 \}$ if $\mathsf{d}(r) \le k < \mathsf{n}(r)$, and
			\vspace{2pt}
			\item $\mathcal{U}_k(S_r) = \{ k + j (\mathsf{n}(r) - \mathsf{d}(r)) \mid j \in \zz_{\ge \ell} \}$ for some $\ell \in \zz_{< 0}$ if $k \ge \mathsf{n}(r)$.
		\end{itemize}
		\vspace{3pt}
		
		\item If $r \in \nn$, then $\mathcal{U}_k(S_r) = \{k\}$ for every $k \in \nn$.
	\end{enumerate}
\end{prop}

\begin{proof}
	That $\mathcal{U}_k(S_r)$ is an arithmetic progression containing $k$ with distance $| \mathsf{n}(r) - \mathsf{d}(r)|$ is an immediate consequence of Theorem~\ref{thm:sets of lengths}.
	
	To show~(1), assume that $r < 1$. Suppose first that $k < \mathsf{n}(r)$. Take $L \in \mathcal{L}(S_r)$ with $k \in L$, and take $x \in S_r$ such that $L = \mathsf{L}(x)$. Choose $z = \sum_{i=0}^N \alpha_i r^i \in \mathsf{Z}(x)$ with $\sum_{i=0}^N \alpha_i = k$. Since $\alpha_i \le k < \mathsf{n}(r)$ for $i \in \llb 0, N \rrb$, Lemma~\ref{lem:factorization of extremal length II} ensures that $|\mathsf{Z}(x)| = 1$, which yields $L = \mathsf{L}(x) = \{k\}$. Thus, $\mathcal{U}_k(S_r) = \{k\}$. Now suppose that $\mathsf{n}(r) \le k < \mathsf{d}(r)$. Notice that the element $k \in S_r$ has a factorization of length $k$, namely, $k \cdot 1 \in \mathsf{Z}(k)$. Now we can use Lemma~\ref{lem:factorization of extremal length II}(3) to conclude that $\sup \mathsf{L}(k) = \infty$. Hence $\rho_k(S_r) = \infty$. On the other hand, let $x$ be an element of $S_r$ having a factorization of length $k$. Since $k < \mathsf{d}(r)$, it follows by Lemma~\ref{lem:factorization of extremal length II}(1) that any length-$k$ factorization in $\mathsf{Z}(x)$ is a factorization of~$x$ of minimum length. Hence $\lambda_k(S_r) = k$ and, therefore,
	\[
		\mathcal{U}_k(S_r) = \{ k + j (\mathsf{d}(r) - \mathsf{n}(r)) \mid j \in \nn_0 \}.
	\]
	Now assume that $k \ge \mathsf{d}(r)$. As $k \ge \mathsf{n}(r)$, we have once again that $\rho_k(S_r) = \infty$. Also, because $k \ge \mathsf{d}(r)$ one finds that $(k - \mathsf{d}(r))r + \mathsf{n}(r) \cdot 1$ is a factorization in $\mathsf{Z}(kr)$ of length $k - (\mathsf{d}(r) - \mathsf{n}(r))$. Then there exists $\ell \in \zz_{< 0}$ such that
	\[
		\mathcal{U}_k(S_r) = \{ k + j (\mathsf{d}(r) - \mathsf{n}(r)) \mid j \in \zz_{\ge \ell} \}.
	\]
	
	Suppose now that $r \in \qq_{> 1} \setminus \nn$. Assume first that $k < \mathsf{d}(r)$. Take $L \in \mathcal{L}(S_r)$ containing $k$ and $x \in S_r$ such that $L = \mathsf{L}(x)$. If $z = \sum_{i=0}^N \alpha_i r^i \in \mathsf{Z}(x)$ satisfies $|z| = k$, then $\alpha_i \le k < \mathsf{d}(r)$ for $i \in \llb 0,N \rrb$, and Lemma~\ref{lem:factorization of extremal length I} implies that $L = \mathsf{L}(x) = \{k\}$. As a result, $\mathcal{U}_k(S_r) = \{k\}$. Suppose now that $\mathsf{d}(r) \le k < \mathsf{n}(r)$. In this case, for each $n > k$, we can consider the element $x_n = k r^n \in S_r$ and set $L_n := \mathsf{L}(x_n)$. It is not hard to check that
	\[
		z_n := \mathsf{n}(r) \cdot 1 + \bigg( \sum_{i=1}^{n-1} \big( \mathsf{n}(r) - \mathsf{d}(r)\big) r^i \bigg) + \big( k - \mathsf{d}(r) \big) r^n
	\]
	is a factorization of $x_n$. Therefore $|z_n| = k + n( \mathsf{n}(r) - \mathsf{d}(r)) \in L_n$. Since $k \in L_n$ for every $n \in \nn$, it follows that $\rho_k(S_r) = \infty$. On the other hand, it follows by Lemma~\ref{lem:factorization of extremal length I}(1) that any factorization of length $k$ of an element $x \in S_r$ must be a factorization of minimum length in $\mathsf{Z}(x)$. Hence $\lambda_k(S_r) = k$, which implies that
	\[
		\mathcal{U}_k(S_r) = \{ k + j (\mathsf{n}(r) - \mathsf{d}(r)) \mid j \in \nn_0 \}.
	\]
	Assume now that $k \ge \mathsf{n}(r)$. As $k \ge \mathsf{d}(r)$ we still obtain $\rho_k(S_r) = \infty$. In addition, because $k \ge \mathsf{n}(r)$, we have that $(k - \mathsf{n}(r)) \cdot 1 + \mathsf{d}(r)r$ is a factorization in $\mathsf{Z}(k)$ having length $k - (\mathsf{n}(r) - \mathsf{d}(r))$. Thus, there exists $\ell \in \zz_{< 0}$ such that
	\[
		\mathcal{U}_k(S_r) = \{ k + j (\mathsf{n}(r) - \mathsf{d}(r)) \mid j \in \zz_{\ge \ell} \}.
	\]
	
	Finally, condition~(3) follows directly from the fact that $S_r = (\nn_0,+)$ when $r \in \nn$ and, therefore, for every $k \in \nn$ there exists exactly one element in $S_r$ having a length-$k$ factorization, namely $k$.
\end{proof}

\begin{cor} \label{cor:local elastiicity}
	Take $r \in \qq_{>0}$ such that $S_r$ is atomic. Then $\rho(S_r) < \infty$ if and only if $\rho_k(S_r) < \infty$ for every $k \in \nn$.
\end{cor}

\begin{proof}
	It follows from \cite[Proposition~1.4.2(1)]{GH06} that $\rho_k(S_r) \le k \rho(S_r)$, which yields the direct implication. For the reverse implication, we first notice that, by Proposition~\ref{local}, if $r \notin \nn$ and $k > \max \{\mathsf{n}(r), \mathsf{d}(r)\}$, then $\rho_k(S_r) = \infty$. Hence the fact that $\rho_k(S_r) < \infty$ for every $k \in \nn$ implies that $r \in \nn$. In this case $\rho(S_r) = \rho(\nn_0) = 1$, and so $\rho(S_r) < \infty$.
\end{proof}

As~\cite[Proposition~1.4.2(1)]{GH06} holds for every atomic monoid, the direct implication of Corollary~\ref{cor:local elastiicity} also holds for any atomic monoid. However, the reverse implication of the same corollary is not true even in the context of Puiseux monoids.

\begin{example}
	Let $\{p_n\}$ be a strictly increasing sequence of primes, and consider the Puiseux monoid
	\[
		M := \bigg\langle \frac{p_n^2 + 1}{p_n} \ \bigg{|} \ n \in \nn \bigg \rangle.
	\]
	It is not hard to verify that the monoid $M$ is atomic with set of atoms given by the displayed generating set. Then it follows from \cite[Theorem~3.2]{GO17} that $\rho(S_r) = \infty$. However, \cite[Theorem~4.1(1)]{mG19} guarantees that $\rho_k(M) < \infty$ for every $k \in \nn$.
\end{example}

\section{The Tame Degree}
\label{sec:tame degree}

\subsection{Omega Primality} Let $M$ be a reduced atomic monoid. The \emph{omega function} $\omega \colon M \to \nn_0 \cup \{\infty\}$ is defined as follows: for each $x \in M^\bullet$ we take $\omega(x)$ to be the smallest $n \in \nn$ satisfying that whenever $x \mid_M \sum_{i=1}^t a_i$ for some $a_1, \dots, a_t \in \mathcal{A}(M)$, there exists $T \subseteq \llb 1, t \rrb$ with $|T| \le n$ such that $x \mid_M \sum_{i \in T} a_i$.  If no such $n$ exists, then $\omega(x) = \infty$. In addition, we define $\omega(0) = 0$. Then we define
\[
	\omega(M) := \sup\{\omega(a) \mid a \in \mathcal{A}(M)\}.
\]
Notice that $\omega(x) = 1$ if and only if $x$ is prime in $M$. The omega function was introduced by Geroldinger and Hassler in~\cite{GH08} to measure how far in an atomic monoid an element is from being prime.

Before proving the main results of this section, let us collect two technical lemmas.

\begin{lemma} \label{lem:element divisible by 1}
	If $r \in \qq_{> 1}$, then $1 \mid_{S_r} \mathsf{d}(r) r^k$ for every $k \in \nn_0$.
\end{lemma}

\begin{proof}
	If $r \in \nn$, then $S_r = (\nn_0,+)$ and the statement of the lemma follows straightforwardly. Then we assume that $r \in \qq_{> 1} \setminus \nn$. For $k=0$, the statement of the lemma holds trivially. For $k \in \nn$, consider the factorization $z_k := \mathsf{d}(r) \, r^k \in \mathsf{Z}(S_r)$. The factorization
	\[
		z := \mathsf{n}(r) + \sum_{i=1}^{k-1} (\mathsf{n}(r) - \mathsf{d}(r)) r^i
	\]
	belongs to $\mathsf{Z}(\phi(z_k))$ (recall that $\phi \colon \mathsf{Z}(S_r) \to S_r$ is the factorization homomorphism of~$S_r$). This is because
	\begin{align*}
		\mathsf{n}(r) + \sum_{i=1}^{k-1} (\mathsf{n}(r) - \mathsf{d}(r)) r^i
			&= \mathsf{n}(r) + \sum_{i=1}^{k-1} \mathsf{n}(r) r^i - \sum_{i=1}^{k-1} \mathsf{d}(r) r^i \\ &= \mathsf{n}(r) + \sum_{i=1}^{k-1} \mathsf{n}(r) r^i - \sum_{i=1}^{k-1} \mathsf{n}(r) r^{i-1} = \mathsf{d}(r) r^k.
	\end{align*}
	Hence $1 \mid_{S_r} \mathsf{d}(r) r^k$
\end{proof}

\begin{lemma} \label{lem:1 dividies constant coefficient of min-length factorization}
	Take $r \in \qq \cap (0,1)$ such that $S_r$ is atomic, and let $\sum_{i=0}^N \alpha_i r^i$ be the factorization in $\mathsf{Z}(x)$ of minimum length. Then $\alpha_0 \ge 1$ if and only if $1 \mid_{S_r} x$.
\end{lemma}

\begin{proof}
	The direct implication is straightforward. For the reverse implication, suppose that $1 \mid_{S_r} x$. Then there exists a factorization $z' := \sum_{i=0}^K \beta_i r^i \in \mathsf{Z}(x)$ such that $\beta_0 \ge 1$. If $\beta_i \ge \mathsf{d}(r)$ for some $i \in \llb 1,K \rrb$, then we can use the identity $\mathsf{d}(r) r^i = \mathsf{n}(r) r^{i-1}$ to find another factorization $z'' \in \mathsf{Z}(x)$ such that $|z''| < |z'|$. Notice that the atom $1$ appears in~$z''$. Then we can replace $z'$ by $z''$. After carrying out such a replacement as many times as possible, we can guarantee that $\beta_i < \mathsf{d}(r)$ for $i \in \llb 1,K \rrb$. Then Lemma~\ref{lem:factorization of extremal length II}(1) ensures that $z'$ is a minimum-length factorization of $x$. Now Lemma~\ref{lem:factorization of extremal length II}(2) implies that $z' = z$. Finally, $\alpha_0 = \beta_0 \ge 1$ follows from the fact that the atom $1$ appears in~$z'$.
\end{proof}

\begin{prop} \label{prop:omega primality}
	Take $r \in \qq_{> 0}$ such that  $S_r$ is atomic.
	\begin{enumerate}
		\item If $r<1$, then $\omega(1) = \infty$.
		\vspace{3pt}
		
		\item If $r \in \nn$, then $\omega(1) = 1$.
		\vspace{3pt}
		
		\item If $r \in \qq_{> 1} \setminus \nn$, then $\omega(1) = \mathsf{d}(r)$.
	\end{enumerate}
\end{prop}

\begin{proof}
	To verify~(1), suppose that $r < 1$. Then set $x = \mathsf{n}(r) \in S_r$ and note that $1 \mid_{S_r} x$. Fix an arbitrary $N \in \nn$. Take now $n \in \nn$ such that $\mathsf{d}(r) + n( \mathsf{d}(r) - \mathsf{n}(r)) \ge N$. It is not hard to check that
	\[
		z := \mathsf{d}(r) r^{n+1} + \sum_{i=1}^n (\mathsf{d}(r) - \mathsf{n}(r) ) r^i
	\]
	is a factorization in $\mathsf{Z}(x)$. Suppose that $z' = \sum_{i=1}^K \alpha_i r^i$ is a sub-factorization of $z$ such that $1 \mid_{S_r} x' := \phi(z')$. Now we can move from $z'$ to a factorization $z''$ of $x'$ of minimum length by using the identity $\mathsf{d}(r)r^{i+1} = \mathsf{n}(r)r^i$ finitely many times. As $1 \mid_{S_r} x'$, it follows by Lemma~\ref{lem:1 dividies constant coefficient of min-length factorization} that the atom $1$ appears in $z''$. Therefore, when we obtained $z''$ from~$z'$ (which does not contain $1$ as a formal atom), we must have applied the identity $\mathsf{d}(r)r = \mathsf{n}(r) \cdot 1$ at least once. As a result $z''$ contains at least $\mathsf{n}(r)$ copies of the atom~$1$. This implies that $x' = \phi(z'') \ge \mathsf{n}(r) = x$. Thus, $x' = x$, which implies that $z'$ is the whole factorization~$z$. As a result, $\omega(1) \ge |z| \ge N$. Since $N$ was arbitrarily taken, we can conclude that $\omega(1) = \infty$, as desired.
	
	Notice that~(2) is a direct consequence of the fact that $1$ is a prime element in $S_r = (\nn_0,+)$.
	
	Finally, we prove~(3). Take $z = \sum_{i=0}^N \alpha_i r^i \in \mathsf{Z}(x)$ for some $x \in S_r$ such that $1 \mid_{S_r} x$. We claim that there exists a sub-factorization $z'$ of $z$ such that $|z'| \le \mathsf{d}(r)$ and $1 \mid_{S_r} \phi(z')$, where $\phi$ is the factorization homomorphism of $S_r$. If $\alpha_0 > 0$, then $1$ is one of the atoms showing in $z$ and our claim follows trivially. Therefore assume that $\alpha_0 = 0$. Since $1 \mid_{S_r} x$ and $1$ does not show in $z$, we have that $|\mathsf{Z}(x)| > 1$. Then conditions~(1) and~(3) in Lemma~\ref{lem:factorization of extremal length I} cannot be simultaneously true, which implies that $\alpha_i \ge \mathsf{d}(r)$ for some $i \in \llb 1, N \rrb$. Lemma~\ref{lem:element divisible by 1} ensures now that $1 \mid_{S_r} \phi(z')$ for the sub-factorization $z' := \mathsf{d}(r)r^i$ of $z$. This proves our claim and implies that $\omega(1) \le \mathsf{d}(r)$. On the other hand, take $w$ to be a strict sub-factorization of $\mathsf{d}(r) \, r$. Note that the atom $1$ does not appear in $w$. In addition, it follows by Lemma~\ref{lem:factorization of extremal length I} that $|\mathsf{Z}(\phi(w))| = 1$. Hence $1 \nmid_{S_r} \phi(w)$. As a result, we have that $\omega(1) \ge \mathsf{d}(r)$, and~(3) follows.
\end{proof}


\subsection{Tameness} For an atom $a \in \mathcal{A}(M)$, the {\it local tame degree} $\mathsf{t}(a) \in \nn_0$ is the smallest $n \in \nn_0 \cup \{\infty\}$ such that in any given factorization of $x \in a + M$  at most $n$ atoms have to be replaced by at most $n$ new atoms to obtain a new factorization of $x$ that contains $a$. More specifically, it means that $\mathsf{t}(a)$ is the smallest $n \in \nn_0 \cup \{\infty\}$ with the following property: if $\mathsf{Z}(x) \cap (a + \mathsf{Z}(M)) \ne \emptyset$ and $z \in \mathsf{Z}(x)$, then there exists a $z' \in \mathsf{Z}(x) \cap (a + \mathsf{Z}(M))$ such that $\mathsf{d}(z,z') \le n$.

\begin{definition}
	An atomic monoid $M$ is said to be {\it locally tame} provided that $\mathsf{t}(a) < \infty$ for all $a \in \mathcal{A}(M)$.
\end{definition}

\noindent Every factorial monoid is locally tame (see \cite[Theorem~1.6.6 and Theorem~1.6.7]{GH06}). In particular, $(\nn_0,+)$ is locally tame. The tame degree of numerical monoids was first considered in~\cite{CGL09}. 
The factorization invariant $\tau \colon M \to \nn_0 \cup \{\infty\}$, which was introduced in~\cite{GH08}, is defined as follows: for $k \in \nn$ and $b \in M$, we take
\[
	\mathsf{Z}_{\text{min}}(k,b) := \bigg\{ \sum_{i=1}^j a_i \in \mathsf{Z}(M) \ \bigg{|} \ j \le k, \ b \mid_M \sum_{i=1}^j a_i, \, \text{ and } \, b \nmid_M \sum_{i \in I} a_i \ \text{ for any } \ I \subsetneq \llb 1,j \rrb \bigg\}
\]
and then we set
\[
	\tau(b) = \sup_k \sup_z \big\{ \min \mathsf{L}\big(\phi(z) - b \big) \mid z \in \mathsf{Z}_{\text{min}}(k,b)\big\}.
\]
The monoid $M$ is called {\it (globally) tame} provided that the \emph{tame degree}
\[
	\mathsf{t}(M) = \sup \{\mathsf{t}(a) \mid a \in \mathcal{A}(M)\} < \infty.
\]

The following result will be used in the proof of Theorem~\ref{thm:cyclic semirings are no locally tame}.

\begin{theorem} \cite[Theorem~3.6]{GH08} \label{thm:characterization of locally tame monoids}
	Let $M$ be a reduced atomic monoid. Then $M$ is locally tame if and only if $\omega(a) < \infty$ and $\tau(a) < \infty$ for all $a \in \mathcal{A}(M)$.
\end{theorem}

We conclude this section by characterizing the cyclic rational semirings that are locally tame.

\begin{theorem} \label{thm:cyclic semirings are no locally tame}
	Take $r \in \qq_{>0}$ such that $S_r$ is atomic. Then the following conditions are equivalent:
	\begin{enumerate}
		\item $r \in \nn$;
		\vspace{3pt}
		\item $\omega(S_r) < \infty$;
		\vspace{3pt}
		\item $S_r$ is globally tame;
		\vspace{3pt}
		\item $S_r$ is locally tame.
	\end{enumerate}
\end{theorem}

\begin{proof}
	That (1) implies (2) follows from Proposition~\ref{prop:omega primality}(2). Now suppose that (2) holds. Then~\cite[Proposition~3.5]{GK10} ensures that $\mathsf{t}(S_r) \le \omega(S_r)^2 < \infty$, which implies~(3). In addition, (3) implies (4) trivially. 
	
	To prove that (4) implies (1) suppose, by way of contradiction, that $r \in \qq_{> 0} \setminus \nn$. Let us assume first that $r < 1$. In this case, $\omega(1) = \infty$ by Proposition~\ref{prop:omega primality}(3). Then it follows by Theorem~\ref{thm:characterization of locally tame monoids} that $S_r$ is not locally tame, which is a contradiction. For the rest of the proof, we assume that $r \in \qq_{> 1} \setminus \nn$.
	
	We proceed to show that $\tau(1) = \infty$. For $k \in \nn$ such that $k \ge \mathsf{d}(r)$, consider the factorization $z_k = \mathsf{d}(r) r^k \in \mathsf{Z}(S_r)$. Since any strict sub-factorization $z'_k$ of $z_k$ is of the form $\beta r^k$ for some $\beta < \mathsf{d}(r)$, it follows by Lemma~\ref{lem:factorization of extremal length I} that $|\mathsf{Z}(z'_k)| = 1$. On the other hand, $1 \mid_{S_r} \mathsf{d}(r) r^k$ by Lemma~\ref{lem:element divisible by 1}. Therefore $z_k \in \mathsf{Z}_{\text{min}}(k, 1)$. Now consider the factorization
	\[
		z'_k := (\mathsf{n}(r) - 1) \cdot 1 + \sum_{i=1}^{k-1} (\mathsf{n}(r) - \mathsf{d}(r)) r^i.
	\]
	Proceeding as in the proof of Lemma~\ref{lem:element divisible by 1}, one can verify that $\phi(z'_k) = \mathsf{d}(r)r^k - 1$. In addition, the coefficients of the atoms $1, \dots, r^{k-1}$ in $z'_k$ are all strictly less than~$\mathsf{n}(r)$. Then it follows from Lemma~\ref{lem:factorization of extremal length I}(1) that $z'_k$ is a factorization of $\mathsf{d}(r)r^k - 1$ of minimum length. Because $|z'_k| = k(\mathsf{n}(r) - \mathsf{d}(r)) + \mathsf{d}(r) - 1$, one has that
	\begin{align*}
		\tau(1) 
			&= \sup_k \sup_z \big\{ \min \mathsf{L}\big(\phi(z) - 1 \big) \mid z \in \mathsf{Z}_{\text{min}}(k,1)\big\} \\
			&\ge \sup_k \min \mathsf{L}\big( \phi(z_k) - 1 \big) = \sup_k |z'_k| \\
			&= \lim_{k \to \infty} k(\mathsf{n}(r) - \mathsf{d}(r)) + \mathsf{d}(r) - 1 \\
			&= \infty.
	\end{align*}
Hence $\tau(1) = \infty$. Then it follows by Theorem~\ref{thm:characterization of locally tame monoids} that $S_r$ is not locally tame, which contradicts condition~(3). Thus, (3) implies (1), as desired.
\end{proof}
\medskip

\section{Summary}

We close in Table~\ref{Table 3} with a comparison between the various factorization invariants we have studied for a Puiseux monoid $S_r := \langle r^n \mid n \in \nn_0 \rangle$ generated by a geometric sequence and those for a numerical monoid generated by an arithmetic sequence, namely,
\[
	N := \langle n, n+d, \dots, n+kd \rangle,
\]
where $n$, $d$, and $k$ are positive integers with $k \le n-1$. Note that the corresponding results we obtain for the monoid $S_r$ were obtained for the monoid $N$ in the series of five papers \cite{ACHP07,ACKT11,BCKR,CGL09,CHM}, which appeared over a five-year period (2006--2011).

{\footnotesize
	
 \begin{table}[t] \caption{Monoidal Factorization Invariant Comparison}\label{Table 3}
		 	\begin{tabular}{ | p{7cm} | p{7cm}  |} \hline
		 	 \rowcolor{bleudefrance} \textbf{Numerical monoids of the form} $\mathbf{N=\langle n, n+d, \dots, n+kd\rangle}$ & \textbf{Puiseux monoids of the form} $\mathbf{S_r= \big\langle r^n \mid n \in \nn_0 \rangle}$ \\ \hline
		\rowcolor{blizzardblue}\multicolumn{2}{|c|}{System of sets of lengths}\\ \hline
		Sets of lengths in $N$ are arithmetic progressions \cite[Thm.~3.9]{BCKR} \cite[Thm.~2.2]{ACHP07}. By these results, $\Delta(N)=\{d\}$. &
		Sets of lengths in $S_r$ are arithmetic progressions (Theorem~\ref{thm:sets of lengths}). A a consequence, $\Delta(S_r)= \{| \mathsf{n}(r)-\mathsf{d}(r)| \}$.
		 \\ \hline
		\rowcolor{blizzardblue} \multicolumn{2}{|c|}{Elasticity}\\ \hline		
		$\rho(N) = \frac{n+dk}{n}$ is accepted~\cite[Thm.~2.1]{CHM} and fully elastic only when $N = \nn_0$~\cite[Thm.~2.2]{CHM}.
		& If $S_r$ is atomic, then $\rho(S_r) \in \{1,\infty\}$ (Corollary \ref{cor:elasticity of rational semirings}). Moreover, $\rho(M)$ is accepted if and only if $r < 1$ or $r \in \nn$ (Proposition \ref{accepted}). $S_r$ is fully elastic when $\mathsf{n}(r) = \mathsf{d}(r) + 1$ (Proposition~\ref{prop:set of elasticities}).
		\\ \hline
		\rowcolor{blizzardblue} \multicolumn{2}{|c|}{Catenary degree}\\ \hline
		$\mathsf{c}(N) = \left\lceil \frac{n}{k}\right\rceil +d$ \cite[Thm.~14]{CGL09} &
	 	If $S_r$ is atomic, then $\mathsf{c}(S_r) = \max\{\mathsf{n}(r), \mathsf{d}(r)\}$ (Corollary \ref{bigcor}) \\ \hline
		\rowcolor{blizzardblue} \multicolumn{2}{|c|}{Tame degree}\\ \hline
		$N$ is always globally tame (and, consequently, locally tame) \cite[Thm. 3.1.4]{GH06}. &
		$S_r$ is globally tame if and only if $S_r$ is locally tame if and only if $r \in \nn$. (Theorem \ref{thm:cyclic semirings are no locally tame}). \\ \hline
		\rowcolor{blizzardblue} \multicolumn{2}{|c|}{Omega primality}\\ \hline
		$\omega(N)=\infty$ \cite[Prop.~2.1]{ACKT11}.
		&
		If $S_r$ is atomic and $r<1$, then $\omega(S_r) = \infty$ (Theorem~\ref{thm:cyclic semirings are no locally tame}).\\ \hline
  	\end{tabular} \end{table}
}
\bigskip

\section*{Acknowledgements}
While working on this paper, the second author was supported by the UC Year Dissertation Fellowship. The authors are grateful to an anonymous referee for helpful suggestions.


\bigskip

\begin{thebibliography}{20}
	
	\bibitem{ACHP07} J. Amos, S. T. Chapman, N. Hine, and J. Paixao: \emph{Sets of lengths do not characterize numerical monoids}, Integers {\bf 7} (2007) A50.
	
	\bibitem{ACKT11} D.~F. Anderson, S.~T. Chapman, N. Kaplan, and D. Torkornoo: \emph{An algorithm to compute $\omega$-primality in a numerical monoid}, Semigroup Forum {\bf 82} (2011) 96--108. 

	\bibitem{AG16} A. Assi and P.~A. Garc\'ia-S\'anchez: \emph{Numerical Semigroups and Applications}, RSME Springer Series, Springer, New York, 2016.
	
	\bibitem{BS18} N. Baeth and D. Smertnig: \emph{Arithmetical invariants of local quaternion orders}, Acta Arith. {\bf 186} (2018) 143--177.

	\bibitem{BOP17b} T. Barron, C. O'Neill, and R. Pelayo: \emph{On dynamic algorithms for factorization invariants in numerical monoids}, Math. Comp. {\bf 86} (2017) 2429--2447.
	
	\bibitem{BCKR} C. Bowles, S. Chapman, N. Kaplan, and D. Reiser: \emph{On delta sets of numerical monoids}, J. Algebra Appl. {\bf 5} (2006) 695--718.
	
	\bibitem{BGG11} V. Blanco, P.~A. Garc\'ia-S\'anchez, and A. Geroldinger: \emph{Semigroup-theoretical characterizations of arithmetical invariants with applications to numerical monoids and Krull monoids}, Illinois J. Math. {\bf 55} (2011) 1385--1414.

	\bibitem{CCMMP17} S.~T. Chapman, M. Corrales, A. Miller, C. Miller, and D. Patel: \emph{The catenary degrees of elements in numerical monoids generated by arithmetic sequences}, Comm. Algebra {\bf 45} (2017) 5443--5452.
	
	\bibitem{CDHK10} S.~T. Chapman, J. Daigle, R. Hoyer, and N. Kaplan: \emph{Delta sets of numerical monoids using nonminimal sets of generators}, Comm. Algebra {\bf 38} (2010) 2622--2634.
	
	\bibitem{CGL09} S.~T.~Chapman, P.~A.~Garc\'ia-S\'anchez, and D.~Llena: \emph{The catenary and tame degree of numerical monoids}, Forum Math. {\bf 21} (2009) 117--129.
	
	\bibitem{CHM} S. Chapman, M. Holden, T. Moore: \emph{Full elasticity in atomic monoids and integral domains}, Rocky Mountain J. Math. {\bf 36} (2006) 1437--1455.

	\bibitem{CHK09} S.~T. Chapman, R. Hoyer, and N. Kaplan: \emph{Delta sets of numerical monoids are eventually periodic}, Aequationes Math. {\bf 77} (2009) 273--279.
	
	\bibitem{CKLNZ14} S.~T. Chapman, N. Kaplan, T. Lemburg, A. Niles, and C. Zlogar: \emph{Shifts of generators and delta sets of numerical monoids}, Internat. J. Algebra Comput. {\bf 24} (2014) 655--669.
	
	\bibitem{CK17} S. Colton and N. Kaplan: \emph{The realization problem for delta sets of numerical semigroups}, J. Commut. Algebra {\bf 9} (2017) 313--339.
	
	\bibitem{FGKT17} Y. Fan, A. Geroldinger, F. Kainrath, and S. Tringali: \emph{Arithmetic of commutative semigroups with a focus on semigroups of ideals and modules}, J. Algebra Appl. {\bf 11} (2017) 1750234.
	
	\bibitem{GGMV2} J. Garc\'ia-Garc\'ia, M. Moreno-Fr\'ias, and A. Vigneron-Tenorio: \emph{Computation of the $\omega$-primality and asymptotic $\omega$-primality with applications to numerical semigroups}, Israel J. Math. {\bf 206} (2015) 395--411.
	
	\bibitem{GOW19}  P. Garc\'ia-S\'anchez, C. O'Neill, and G. Webb: \emph{On the computation of factorization invariants for affine semigroups}, J. Algebra Appl. {\bf 18} (2019) 1950019.

	\bibitem{GR09} P.~A.~Garc\'ia-S\'anchez and J.~C.~Rosales: \emph{Numerical Semigroups}, Developments in Mathematics Vol. 20, Springer-Verlag, New York, 2009.

	\bibitem{aG16} A.~Geroldinger: \emph{Sets of Lengths}, Amer. Math. Monthly {\bf 123} (2016) 960--988.

	\bibitem{GH06} A. Geroldinger and F. Halter-Koch: \emph{Non-unique Factorizations: Algebraic, Combinatorial and Analytic Theory}, Pure and Applied Mathematics Vol. 278, Chapman \& Hall/CRC, Boca Raton, 2006.
	
	\bibitem{GH08} A.~Geroldinger and W.~Hassler, \emph{Local tameness of v-noetherian monoids}, J. Pure Appl. Algebra {\bf 212} (2008) 1509--1524.
	
	\bibitem{GK10} A. Geroldinger and F. Kainrath: \emph{On the arithmetic of tame monoids with applications to Krull monoids and Mori domains}, J. Pure Appl. Algebra {\bf 214} (2010) 2199--2218.

	\bibitem{GS18} A. Geroldinger and W. Schmid: \emph{A realization theorem for sets of lengths in numerical monoids}, Forum Math. {\bf 30} (2018) 1111--1118.
	
	\bibitem{GSZ17} A. Geroldinger, W.~A. Schmid, and Q. Zhong: (2017) \emph{Systems of Sets of Lengths: Transfer Krull Monoids Versus Weakly Krull Monoids}. In: M. Fontana, S. Frisch, S. Glaz, F. Tartarone, P. Zanardo (eds) Rings, Polynomials, and Modules. Springer, Cham.
	
	\bibitem{GZ19} A. Geroldinger and Q. Zhong: \emph{Sets of arithmetical invariants in transfer Krull monoids}, J. Pure Appl. Algebra {\bf 223} (2019) 3889--3918.

	\bibitem{fG19a} F.~Gotti: \emph{Increasing positive monoids of ordered fields are FF-monoids}, J. Algebra {\bf 518} (2019) 40--56.
	
	\bibitem{fG19d} F. Gotti: \emph{Irreducibility and factorizations in monoid rings}, Springer INdAM Series: Proceedings of the IMNS (to appear). [arXiv:1905.07168]
	
	\bibitem{fG17} F. Gotti: \emph{On the atomic structure of Puiseux monoids}, J. Algebra Appl. {\bf 16} (2017) 1750126.
	
	\bibitem{fG19c} F. Gotti: \emph{On the system of sets of lengths and the elasticity of submonoids of a finite-rank free commutative monoid}, J. Algebra Appl., DOI: 10.1142/S0219498820501376. [arXiv:1806.11273]
	
	\bibitem{fG18} F. Gotti: \emph{Puiseux monoids and transfer homomorphisms}, J. Algebra {\bf 516} (2018) 95--114.
	
	\bibitem{fG19b} F. Gotti: \emph{Systems of sets of lengths of Puiseux monoids}, J. Pure Appl. Algebra {\bf 223} (2019) 1856--1868.

	\bibitem{GG17} F.~Gotti and M.~Gotti: \emph{Atomicity and boundedness of monotone Puiseux monoids}, Semigroup Forum {\bf 96} (2018) 536--552.
	
	\bibitem{GO17} F. Gotti and C. O'Neil: \emph{The elasticity of Puiseux monoids}, J. Commut. Algebra, DOI: https://projecteuclid.org/euclid.jca/1523433696. [arXiv:1703.04207]
	
	\bibitem{mG19} M. Gotti: \emph{On the local k-elasticities of Puiseux monoids}, Internat. J. Algebra Comput. {\bf 29} (2019) 147--158.
	
	\bibitem{pG01} P.~A. Grillet: \emph{Commutative Semigroups}, Advances in Mathematics Vol. 2, Kluwer Academic Publishers, Boston, 2001.
	
	\bibitem{HK84} F. Halter-Koch: \emph{On the factorization of algebraic integers into irreducibles}, Coll. Math. Soc. J\'anos Bolyai {\bf 34} (1984) 699--707. 
	
	\bibitem{K} F. Kainrath: \emph{Factorization in Krull monoids with infinite class group}, Colloq. Math. \textbf{80} (1999) 23--30.
	
	\bibitem{OP18} C. O'Neill and R. Pelayo: \emph{Realizable sets of catenary degrees of numerical monoids}, Bull. Australian Math. Soc. \textbf{97} (2018) 240--245.
	
	\bibitem{N71} W. Narkiewicz: \emph{Some unsolved problems}, Bull. Soc. Math. France {\bf 25} (1971) 159--164.
	
	\bibitem{N79} W. Narkiewicz: \emph{Finite abelian groups and factorization problems}, Colloq. Math. {\bf 42} (1979) 319--330.

	\bibitem{OP18} C. O'Neill and R. Pelayo: \emph{Realizable sets of catenary degrees of numerical monoids}, Bull. Aust. Math. Soc. {\bf 97} (2018) 240--245.
	
	\bibitem{Sk76} L. Skula: \emph{On $c$-semigroups}, ACTA Arith. {\bf 31} (1976) 247--257.
	
	\bibitem{sT19} S. Tringali: \emph{Structural properties of subadditive families with applications to factorization theory}, Israel J. Math. (to appear). [arXiv:1706.03525]
	
	\bibitem{V90} R. Valenza: \emph{Elasticity of factorization in number fields}, J. Number Theory {\bf 36} (1990) 212--218.
	
	\bibitem{Z76} A. Zaks: \emph{Half-factorial domains}, Bull. Amer. Math. Soc. {\bf 82} (1976) 721--723.
	
	\bibitem{Z80} A. Zaks: \emph{Half-factorial domains}, Israel J. Math. {\bf 37} (1980) 281--302.
	
	\bibitem{qZ18} Q.~Zhong: \emph{On elasticities of locally finitely generated monoids}. [arXiv:1807.11523]
	
\end{thebibliography}
\end{document}